\RequirePackage{fix-cm}
\documentclass[smallextended]{svjour3}       
\smartqed  
\usepackage{graphicx}
\usepackage{amsfonts}
\usepackage{amsmath, amscd, amssymb, color}
\usepackage[bookmarksnumbered, colorlinks, plainpages]{hyperref}
\usepackage[numbers,sort&compress]{natbib}
\usepackage{dsfont}
\usepackage[numbers,sort&compress]{natbib}
\begin{document}

\title{PHASE-ISOMETRIES ON THE UNIT SPHERE OF $C(K)$
\thanks{The first author is supported by the Natural Science Foundation of China (Grant Nos. 11371201, 11201337, 11201338)}
}

\author{Dongni Tan         \and
        Yueli Gao 
}


\institute{Dongni Tan \at
              Tianjin University of Technology \\
              \email{tandongni0608@sina.cn}           
           \and
           Yueli Gao \at
              Tianjin University of Technology
}

\date{Received: xxxxxx / Accepted: zzzzzz}

\maketitle

\begin{abstract}
We say that a map $T: S_X\rightarrow S_Y$ between the unit spheres of two real normed-spaces $X$ and $Y$ is a phase-isometry if it satisfies
\begin{eqnarray*}
\{\|T(x)+T(y)\|, \|T(x)-T(y)\|\}=\{\|x+y\|, \|x-y\|\}
\end{eqnarray*}
for all $x,y\in S_X$. In the present paper, we show that there is a phase function $\varepsilon:S_X\rightarrow \{-1,1\}$ such that $\varepsilon \cdot T$ is an isometry which can be extended a real linear isometry on $X$ whenever $T$ is surjective, $X=C(K)$ and $Y$ is an arbitrary Banach space. Additionally, if $T$ is a surjective phase-isometry between the unit spheres of $C(K)$ and $C(\Omega)$, where $K$ and $\Omega$ are compact Hausdorff spaces, we prove that there are a homeomorphism $\varphi: \Omega\rightarrow K$ and a continuous unimodular function $h$ on $\Omega$ such that $$T(f)\in\{h\cdot f\circ \varphi,-h\cdot f\circ \varphi\}$$ for all $f\in S_{C(K)}$. This also can be seen as a Banach-Stone type representation for phase-isometries in $C(K)$ spaces.

%
\end{abstract}

\section{Introduction}
\label{intro}

Let $X$ and $Y$ be real normed spaces, and let $A$ and $B$ be subsets of $X$ and $Y$ respectively. A map $T:A\rightarrow B$ is called an {\it isometry}
if  \begin{eqnarray*}
\|T(x)-T(y)\|=\|x-y\|,   \quad (x,y\in A)
\end{eqnarray*}
and a {\it phase-isometry} if  \begin{eqnarray}\label{eq}
\{\|T(x)+T(y)\|, \|T(x)-T(y)\|\}=\{\|x+y\|, \|x-y\|\},   \quad (x,y\in A).
\end{eqnarray}
 Two maps $S, T:A\rightarrow B$ are said to be {\it phase-equivalent} if there is a phase function $\varepsilon:A\rightarrow \{-1,1\}$  such that $\varepsilon \cdot T=S$. It
should be noted that $\varepsilon$ does not need to be continuous. In the present paper, it is the connection between phase-isometries and isometries restrained on the unit spheres that will concern us. We explore the following

 \begin{question}
  Is a phase-isometry between the unit spheres of two real normed spaces $X$ and $Y$ necessarily phase-equivalent to an isometry which can be extended to a linear isometry from $X$ onto $Y$?
 \end{question}

   This question is motivated by Tingley's problem and Wigner's theorem.

 Tingley's problem asked whether every isometry between the unit spheres of real normed spaces can be extended to a linear isometry between the corresponding entire spaces. This problem raised by D. Tingley \cite{Ti} in 1987 has attracted many people's attention. Although it is still open in general case, there is a large number of papers dealing with this topic (Zentralblatt Math. shows 61 related papers published from 2002 to
2020). Please see \cite{T1} and the survey \cite{D} for classical Banach spaces, and see the survey \cite{P18} which contains a good description of non-commutative operator algebras. For some of the most recent papers one can see \cite{MO18,CS,K,CP,CP2,FJ,P2}.

 Wigner's theorem plays a fundamental role in mathematical foundations of quantum
mechanics, and there are many equivalent forms of Wigner's theorem (see the survey \cite{C} for details on Wigner's theorem and its generalisations). One of these is described as follows:  Let $H$ and $K$ be complex or real inner product spaces, and let $T: H \rightarrow K $ be a map. Then $T$ satisfies
\begin{eqnarray}\label{abs}
|\langle T(x), T(y)\rangle|=|\langle x, y \rangle|  \quad (x,y\in H),
\end{eqnarray} if and only if it is phase-equivalent to a linear or an anti-linear isometry, that is,  there exists a phase function $\varepsilon: H \rightarrow \mathbb{K}$ ($\mathbb{K}$ is $\mathbb{C}$ or $\mathbb{R}$) with $|\varepsilon(x)|=1$ such that $\varepsilon\cdot T$ is a linear isometry or a conjugate linear isometry. If $H$ and $K$ are real, it is easily checked that \eqref{abs} implies \eqref{eq}. The converse i.e., \eqref{eq} implies \eqref{abs} follows directly from polarization identity (see \cite[Theorem 2]{MP}).  The previous Wigner's theorem thus establishes that every phase-isometry between two real inner product spaces is phase-equivalent to a real linear isometry. It is natural to ask whether this result remains true when $X$ and $Y$ are real normed but not inner product spaces.

 The first author and Huang in \cite{HT} first attacked this problem for atomic $L_{p}$-spaces with $(p>0)$ and got a positive answer. More general results offering affirmative answers have been
proven in \cite[Theorem 1]{HT} whenever $T$ is surjective and $X$ is smooth, in
\cite[Theorems 9 and 11]{RD} whenever $T$ is surjective and dim $X=2$, or $T$ is surjective and $X$ is
strictly convex. In \cite[Theorem 2.4]{IT} the same result has been proven without
the assumption of surjectivity, assuming only that $Y$ is strictly convex. Very recently, this problem has been completely solved with a positive answer in \cite[Theorem 4.2]{IBT} without any assumptions on $X$ or $Y$.

In the earlier time, the study of phase-isometries which are defined only on unit spheres has started (see \cite{JH} for $l_p$-spaces ($0<p<\infty$) and \cite{TX} for $\mathcal{L}^\infty(\Gamma)$-type spaces). The present paper continues this topic and considers it in more general case. Meanwhile, we generalize the main result of \cite{TX} to $C(K)$-spaces which are the spaces of continuous functions on a compact Hausdorff space $K$ with the supremum norm. To be precise, we show that every surjective phase-isometry between the unit spheres of $C(K)$ and an arbitrary Banach space $Y$ is phase-equivalent to an isometry which can be extended to a linear isometry from $C(K)$ onto $Y$. In the particular case where the target space $Y$ is also a continuous function space, we give a Banach-Stone type representation for phase-isometries in $C(K)$-spaces.

 Let us mention that although the proof of \cite[Theorem 4.2]{IBT} for phase-isometries on the entire space is quite technical, it cannot be applied to phase-isometries that have domain only on the unit sphere. Compared to the whole space, the unit sphere does not have interior points or appropriate algebraic structure. It is interesting to give a geometric idea for individual spaces to overcome these difficulties.

\section{Preliminaries}
\label{sec:1}
In what follows, all normed spaces are real normed spaces. The letters $X$, $Y$ will denote normed spaces and $X^*, Y^*$ are their dual spaces, respectively. For a real normed space $X$, we denote by $S_X$ its unit sphere.
A subset $F$ of $S_X$ is said to be a \emph{maximal convex subset of $S_X$} if it is not properly contained in any other convex subset of $S_X$. For a compact Hausdorff space $K$, $C(K)$ is the space of continuous functions on $K$ with the max-norm.
Given $t\in K$, we will use the following notation:
\begin{equation*}
F_t:=\{f\in S_{C(K)}: f(t)=1\}.
\end{equation*}
It is easily checked that every maximal convex subset of $S_{C(K)}$ is of the form $F_t$ or $-F_t$ for some $t\in K$.
Given an $x\in S_X$, the star set of $x$ with respect to $S_X$ is defined by
\begin{center}
$\mbox{st}(x):=\{y:y\in S_X,\| y+x\|=2\}$.
\end{center}

 The results of this section valid for general normed spaces may have independent interest. We begin with elementary properties on star sets. The second conclusion of the first lemma was previously shown in \cite[Lemma 3.3]{Ta}. The proof is given for readers' convenience.
 \begin{lemma}\label{lemma:4}
  For every $x\in S_X$, if $\mbox{st}(x)$ is convex, then $\mbox{st}(x)$ is a maximal convex subset of $S_X$. Moreover, if $X$ is separable, every maximal convex subset $F$ of $S_X$ satisfies $F=\mbox{st}(x)$ for some $x\in X$.
   \end{lemma}
\begin{proof}Given $x\in S_X$, let $C\subset S_X$ be convex with $x\in C$. Then for every $y\in C$,  $(x+y)/2\in C$ implies that
$\|x+y\|=2$. Thus $y \in \mbox{st}(x)$, and hence $C\subset \mbox{st}(x)$. This entails the maximality of $\mbox{st}(x)$. For the second conclusion, let $x^*\in S_{X^*}$ be the functional such that $$F=\{x\in S_X: x^*(x)=1\},$$ and let $\{y_n\}$ be dense in $F$. Then $x_0=\sum_{n=1}^{\infty}{2^{-n}}y_n$ is just the point such that $F=\mbox{st}(x_0)$. Indeed, it is clear that $F\subset \mbox{st}(x_0)$. To see the converse, for every $y\in \mbox{st}(x_0)$, since $\|y+x_0\|=2$, the Hahn-Banach theorem gives a functional $z^*\in S_{X^*}$ such that $z^*(y+x_0)=\|y+x_0\|=2$. This implies that $z^*(y)=z^*(x_0)=1$, and thus $z^*(y_n)=1$ for all $n$. Since $\{y_n\}$ is dense in $F$, it follows that $z^*(z)=1$ for all $z\in F$, that is $F\subset F_{z^*}:=\{z\in B_X: z^*(z)=1\}$.  This together with the maximality of $F$ yields $F_{z^*}=F$, and thus $y\in F$. So $\mbox{st}(x_0)\subset F$. The proof is complete.
\end{proof}

The following result for phase-isometries on the whole space has already been shown in \cite[Lemma 2]{TH}. The same result holds for phase-isometries that restricts on the unit spheres with a completely analogous proof. We supply the proof for readers's convenience.

\begin{lemma}\label{lemma:1}
Let $X$ and $Y$ be real normed spaces, and let $T: S_X \rightarrow S_Y$ be a surjective phase-isometry. Then it is injective and $T(-x)=-T(x)$ for all $x\in S_X$. In particular, $T^{-1}$ is also a surjective phase-isometry.
\end{lemma}
\begin{proof}
Since $T$ is surjective, we can find a $y\in S_X$ such that $T(y)=-T(x)$. We deduce from that $T$ is phase-isometry that
\begin{equation*}
 \{\|x+y\|,\|x-y\|\}=\{\|T(x)+T(y)\|, \|T(x)-T(y)\|\}=\{0, \|T(x)-T(y)\|\}.
\end{equation*}
If $\|x-y\|=0$, then $x=y$. This thus implies that $T(x)=0$ which is impossible. So $\|x+y\|=0$, and hence $y=-x$, that is $T(-x)=-T(x)$. Similarly, making use of the fact that $T$ satisfies \eqref{eq} we see that $T$ is injective.
\end{proof}
An easy adaptation of the proof of \cite[Lemma 3.2]{TH1} yields the following result. We give the proof for the sake of completeness.
\begin{lemma}\label{lemma:3}
Let $X$ and $Y$ be normed spaces, and let $T:S_X\rightarrow S_Y$ be a surjective phase-isometry. Then for every $x\in S_X$, $\mbox{st}(x)$ is a maximal convex subset of $S_X$ if and only if $\mbox{st}(T(x))$ is a maximal convex subset of $S_Y$.
\end{lemma}
\begin{proof}
 Given $x\in S_X$, set $y:=T(x)$.  By Lemma \ref{lemma:1}, we see that $T$'s inverse $T^{-1}$ is also a surjective phase-isometry.
 To see our conclusion, it suffices to prove that if $\mbox{st}(x)$ is a maximal convex subset of $S_X$, then so is $\mbox{st}(y)$. Suppose that it is not true. By Lemma \ref{lemma:4}, we know that $\mbox{st}(y)$ is not convex. This hence allows
  the existence of two distinct elements $y_1, y_2\in \mbox{st}(y)$ such that $\|y_1+y_2\|<2$. We set $z_1:=\frac{1}{2}(y_1+y)$ and $z_2:=\frac{1}{2}(y_2+y)$. It is easy to see that $z_1, z_2\in \mbox{st}(y)$ and that they satisfy
\begin{equation*}
 \|z_1+z_2\|<2 \ \mbox{and} \ \|z_2-z_1\|<2.
\end{equation*}
Since $T$ is surjective, there are $x_1,x_2\in S_X$ such that $T(x_1)=z_1$ and $T(x_2)=z_2$. That $T$ is a phase-isometry implies that
\begin{equation*}
x_1, x_2\in \mbox{st}(x) \cup \mbox{st}(-x).
\end{equation*}
Thus $\|x_1+x_2\|=2$ or $\|x_1-x_2\|=2$, which leads to a contradiction that $\|z_1+z_2\|=2 $ or $\|z_1-z_2\|=2$. The proof is complete.
\end{proof}

We present here a version of \cite[Lemma 3.4]{Ta} for our use and convenience. For every subset $A$ of a Banach space, $[A]$ stands for its closed linear span.

\begin{lemma}\label{lemma:9}
Let $X$ and $Y$ be two normed spaces, and let $T: S_X \rightarrow S_Y$ be a surjective phase-isometry. Then for every separable subset $A\subset X$,  there are separable subspaces $X_0\subset X$ and $Y_0\subset Y$ such that $A\subset X_0$ and $T(S_{X_0})=S_{Y_0}$.
\end{lemma}

\begin{proof}
Let $X_1=[A]$ and $Y_1=[T(S_{X_1})]$. Then $Y_1$ is separable. In fact, to see this it is enough to show that $T(S_{X_1})$ is separable. Let $\{x_n\}\subset S_{X_1}$ be a dense subset. We will prove that $\{T(x_n)\}\cup\{-T(x_n)\}$ is also dense in $T(S_{X_1})$.

 For every $y\in T(S_{X_1})$, there is a subsequence $\{x_{n_k}\}$ of $\{x_n\}$ such that $x_{n_k}$ converges to $T^{-1}(y)$ in norm as $k$ goes to $\infty$. Now for every $k\geq 1$, choose $y_k\in S_Y$ such that
 $$y_k= \left
\{ \begin{array}{ll}
T(x_{n_k}) & \mbox{if }\,\|T(x_{n_k})-y\|=\|x_{n_k}-T^{-1}(y)\|; \\
-T(x_{n_k}) & \mbox{if } \, \|T(x_{n_k})-y\|=\|x_{n_k}+T^{-1}(y)\|\neq\|x_{n_k}-T^{-1}(y)\|.
\end{array}
\right.$$
Then it is clear that  $$\|y_k-y\|=\|x_{n_k}-T^{-1}(x)\|.$$ This hence shows that $\{T(x_n)\}\cup\{-T(x_n)\}$ is dense in $T(S_{X_1})$.
 We can define inductively separable subspaces $X_n\subset X$ and $Y_n\subset Y$ such that for all $n\geq 2$,
\begin{equation*}
  X_n=[T^{-1}(S_{Y_{n-1})}] \quad \mbox{and} \quad  Y_n=[T(S_{X_{n}})].
\end{equation*}
Then it is easy to verify that $X_0=\bigcup_{n=1}^{\infty} X_n$ and $Y_0=\bigcup_{n=1}^{\infty} Y_n$ are just the subspaces that we need.
\end{proof}

\section{Phase-isometries from $S_{C(K)}$ onto $S_{C(\Omega)}$}
\label{sec:2}

Although there is a generalized version of the following proposition in next section, we give it here for better understanding phase-isometries between the unit spheres of continuous function spaces.

\begin{proposition}\label{prop:6}
Let $K$ and $\Omega$ be compact Hausdorff spaces, and let $T:S_{C(K)}\rightarrow S_{C(\Omega)}$ be a surjective phase-isometry. Then there is a homeomorphism $\sigma:K\rightarrow \Omega$ such that
 $T(F_t\cup -F_t)=F_{\sigma(t)}\cup-F_{\sigma(t)}$ for every $t\in K$.
\end{proposition}
\begin{proof}
  For every $t\in K$, let $f_1, f_2\cdots,f_n\in F_t$ and set $$A=\mbox{co}\{f_1,f_2,\cdots,f_n\}.$$ By Lemma \ref{lemma:9}, there are separable subspaces $X_0\subset C(K)$ and $Y_0\subset C(\Omega)$ such that $A\subset X_0$ and $T(S_{X_0})=S_{Y_0}$. Since $A\subset S_{X_0}$ is convex, there is a maximal convex subset $F_0$ of $S_{X_0}$ such that $A\subset F_0$. The separability of $X_0$ and Lemma \ref{lemma:4} guarantee that $F_0=\mbox{st}(x_0)$ for some $x_0\in F_0$. By Lemma \ref{lemma:3}, $\mbox{st}(T(x_0))$ is also a maximal convex subset of $S_{Y_0}$. On the other hand, Lemma \ref{lemma:1} together with the fact that $T$ is a surjective phase-isometry establishes that
\begin{equation*}
  T\big(\mbox{st}(x_0)\cup \mbox{st}(-x_0)\big)=\mbox{st}\big(T(x_0)\big)\cup -\mbox{st}\big(T(x_0)\big).
\end{equation*}
This leads immediately to  $$\{T(f_1), T(f_2),\cdots, T(f_n)\}\subset \mbox{st}\big(T(x_0)\big)\cup -\mbox{st}\big(T(x_0)\big).$$
 It follows that there is a $t_0\in\Omega$ such that $|T(f_i)(t_0)|=1$ for all $i=1,2,\cdots,n$. Hence the closed sets $$\Omega_f:=\{s\in\Omega: |T(f)(s)|=1\}$$ for those $f\in F_t$ have the finite intersection property. Since $\Omega$ is compact, there is at least one $s\in \Omega$ such that $|T(f)(s)|=1$ for all $f\in F_t$.  This  proves that $$T(F_t)\subset F_s\cup -F_s.$$ By Lemma \ref{lemma:1}, we have $T(-F_t)\subset F_s\cup -F_s$. Since $T^{-1}: C(\Omega)\rightarrow C(K)$ is also a surjective phase-isometry, we see actually
 \begin{equation}\label{equ:3}
  T(F_t)\cup T(-F_t)=F_s\cup -F_s.
 \end{equation}
Thus the previous argument produces a map $\sigma: K\rightarrow \Omega$ satisfying $\sigma(t)=s$, where $s\in\Omega$ such that \eqref{equ:3} holds for every $t\in K$. It is easily checked from \eqref{equ:3} and the surjectivity of $T$ that $\sigma$ is a bijection. Furthermore, for every $f\in C(K)$, $\sigma$ maps the set $K_f':=\{t\in K: |f(t)|<1\}$ onto the set $\Omega_{T(f)}':=\{s\in \Omega: |T(f)(s)|<1\}$. Since $\{K_f'\}_{f\in S_{C(K)}}$ and $\{\Omega_{T(f)}'\}_{f\in S_{C(K)}}$  form a basis for the topologies of $K$ and $\Omega$ respectively, it follows that $\sigma$ is a homeomorphism. This finishes the proof.
\end{proof}

 Observe that the extreme points of $B_{C(K)}$ are those continuous functions $h$ with $|h(t)|=1$ for all $t\in K$. It should be noted that some $h\in B_{C(K)}$ is an extreme point if and only if its star is maximal. A straightforward application of Proposition \ref{prop:6} is that every phase-isometry from $S_{C(K)}$ onto $S_{C(\Omega)}$ carries extreme points to extreme points.

\begin{lemma}\label{lemma:2}
Let $K$ and $\Omega$ be compact Hausdorff spaces, and let $T:S_{C(K)}\rightarrow S_{C(\Omega)}$ be a surjective phase-isometry. Then $T(e)$ is an extreme point of $B_{C(\Omega)}$ if and only if $e$ is an extreme point of $B_{C(K)}$.
\end{lemma}
For a compact Hausdorff space $K$, we will use the notations
\begin{equation*}
  S_{C(K)^{+}}:=\{f\in S_{C(K)}: f(t)\geq 0 \,\,\mbox{for\, all}\,\, t\in K\}
\end{equation*}
and \begin{equation*}
  S_{C(K)^{-}}:=\{f\in S_{C(K)}: f(t)\leq 0 \,\,\mbox{for\, all}\,\, t\in K\}.
\end{equation*}
A result below characterizes the image of those positive functions of norm one under surjective phase-isometries between the unit spheres of two continuous function spaces.

For all $a,b\in\mathbb{R}$, we set $a\vee b=\max\{a,b\}$, and $a\wedge b=\min \{a,b\}$. The following result provides a Banach-Stone type representation for phase-isometries between the unit spheres of continuous function spaces.
\begin{theorem}\label{lemma:5}
 Let $K$ and $\Omega$ be compact Hausdorff spaces, and let $T:S_{C(K)}\rightarrow S_{C(\Omega)}$ be a surjective phase-isometry. Suppose that $\sigma$ is as in Proposition \ref{prop:6}. Then for every $t\in K$, we have
$$T(f)(\sigma(t))\in \{f(t), -f(t)\}$$ for all $f\in S_{C(K)}$. Moreover, there are  a homeomorphism $\varphi$ from $\Omega$ onto $K$ and a continuous unimodular function $h$ on $\Omega$ such that  for every $f\in S_{C(K)}$, there is a $\theta\in\{-1,1\}$ such that
\begin{equation}\label{equ:11}
T(f)(s)=\theta h(s) f(\varphi(s))  \quad (s\in\Omega).
\end{equation}

\end{theorem}
\begin{proof}
 For the first conclusion, fix $t\in K$, and let $f\in S_{C(K)}$. If $f\in F_t\cup -F_t$, the desired conclusion follows immediately from Proposition \ref{prop:6}. Now for every $f\notin F_t\cup -F_t$, replacing $f$ by $-f$ if necessary, we may assume that $f(t)\geq 0$. For every $0<\varepsilon<1$, there is a continuous function $\phi:K\rightarrow [0,1]$ such that
  $\phi(s)=1$ if $|f(s)-f(t)|\leq \varepsilon/2$
 and $\phi(s)=0$ if $|f(s)-f(t)|>\varepsilon$. Set
 \begin{equation*}
 \psi(s)=f(s)+\phi(s)(1-f(t)).
\end{equation*}
It follows that $\psi\in F_t$ and $$\|f-\psi\|\leq1-f(t)+\varepsilon.$$

Since $T$ is a phase-isometry, we deduce from this and Proposition \ref{prop:6} that
 \begin{align*}
 1-f(t)+\varepsilon\geq\|f-\psi\|\wedge\|f+\psi\|&=\|T(f)-T(\psi)\|\wedge \|T(f)+T(\psi)\|\\
 &\geq 1-T(f)(\sigma(t))\vee-T(f)(\sigma(t)).
 \end{align*}
It follows that $f(t)\leq T(f)(\sigma(t))+\varepsilon$ or $f(t)\leq -T(f)(\sigma(t))+\varepsilon$. Since $\varepsilon$ is arbitrary, we see that actually
$f(t)\leq T(f)(\sigma(t))$ or $f(t)\leq -T(f)(\sigma(t))$ respectively.

Note that $T^{-1}$ is also a phase-isometry. Thus a similar argument as above shows that $T(f)(\sigma(t))\leq f(t)$ if $T(f)(\sigma(t))\geq f(t)\geq 0$ or $-T(f)(\sigma(t))\leq f(t)$ if $-T(f)(\sigma(t))\geq f(t)\geq 0$. This entails the first conclusion.

Now let us deal with the second conclusion. Set $h=T(\mathds{1}_K)$ and $\varphi=\sigma^{-1}$. We first show \eqref{equ:11} holds for those $f\in S_{C(K)^{+}}$. Since $|h(s)|=1$ for all $s\in\Omega$ by Lemma \ref{lemma:2}, we get \eqref{equ:11} in this case by combining the first conclusion and
\begin{align*}
 \{\|h-T(f)\|, \|h+T(f)\|\}&=\{\|T(\mathds{1}_K)-T(f)\|, \|T(\mathds{1}_K)+T(f)\|\}\\&=\{\|\mathds{1}_K-f\|,\|\mathds{1}_K+f\|\}=\{2,b\}
\end{align*}
where $0\leq b\leq 1$.  A completely analogous argument shows that \eqref{equ:11} holds for $f\in S_{C(K)^-}$. Proceeding in a similar fashion as the above, we see that \eqref{equ:11} holds for every $f\in S_{C(K)}$ and those $s\in\Omega$ with $|f(\varphi(s))|=1$. i.e., there is $\theta\in\{-1,1\}$ such that
\begin{equation*}
T(f)(s)=\theta h(s) f(\varphi(s))  \quad (s\in\Omega, |f(\varphi(s))|=1).
\end{equation*}
Indeed, for every $f\in S_{C(K)}$ with $f(t_1)=1$ and $f(t_2)=-1$ for distinct $t_1,t_2\in K$, consider an $f_1:K\rightarrow [0,1]$ satisfies $f_1(t)=1$ if $f(t)=1$ and  $f_1(t)=0$ if $f(t)=-1$.

 Since we have
\begin{align*}
 \{\|T(f_1)-T(f)\|, \|T(f_1)+T(f)\|\}=\{\|f_1-f\|,\|f_1+f\|\}=\{2,c\},
\end{align*}
where $1\leq c<2$, it follows that \eqref{equ:11} holds for those $s\in\Omega$ with $f(\varphi(s))=1$. Similarly, we conclude that it holds for those $s\in\Omega$ with $f(\varphi(s))=-1$. These combining with the equation
\begin{align*}
\{\|T(\mathds{1}_K)-T(f)\|, \|T(\mathds{1}_K)+T(f)\|\}=\{\|\mathds{1}_K-f\|,\|\mathds{1}_K+f\|\}=\{2\}
\end{align*}
gives the desired conclusion.

Suppose that \eqref{equ:11} does not hold for all $f\in S_{C(K)}$. Then we can find an $f_0\in S_{C(K)}$ which is not in $S_{C(K)^+}\cup S_{C(K)^-}$ and  $s_1, s_2\in \Omega$ such that
  \begin{equation*}
    T(f_0)(s_1)=h(s_1)f_0(\varphi(s_1))\neq0   \,\,  \mbox{and}  \,\,   T(f_0)(s_2)=-h(s_2)f_0(\varphi(s_2))\neq0.
  \end{equation*}
Put $a:=|f_0(\varphi(s_1))|\wedge |f_0(\varphi(s_2))|$. Let
\begin{equation*}
A:=\{t\in K: f_0(t)\geq-\frac{a}{3}\}\,\,\mbox{and}\,\, B:=\{t\in K: f_0(t)\leq-\frac{a}{2}\}.
\end{equation*}
 If $A$ and $B$ are not empty-sets,  Urysohn's lemma provides a continuous function $g:K\rightarrow [-1,1]$ such that if $t\in A, g(t)=1$ and if $t\in B, g(t)=-1$. If $A$ or $B$ is an empty-set, $g$ is chosen to be the extreme point $-\mathds{1}_K$ or $\mathds{1}_K$ respectively. Since $T$ is a phase-isometry, this together with the above establishes that
\begin{align*}
 1+\frac{a}{2}>\|f_0-g\|\wedge\|f_0+g\| &=\|T(f_0)-T(g)\|\wedge\|T(f_0)+T(g)\|\\&\geq1+|T(f_0)(s_1)|\wedge |T(f_0)(s_2)|=1+a.
\end{align*}
A contradiction thus proves the second conclusion. The proof is complete.
  \end{proof}

As a direct consequence of Theorem \ref{lemma:5}, we have the following corollary.
\begin{corollary}
Let $K$ and $\Omega$ be compact Hausdorff spaces, and let $T:S_{C(K)}\rightarrow S_{C(\Omega)}$ be a surjective phase-isometry. Then $T$ is phase-equivalent to a surjective linear isometry $\Phi$ from $S_{C(K)}$ onto $S_{C(\Omega)}$ of the form $\Phi(f)=f\circ\varphi$ for all $f\in S_{C(K)}$, where $\varphi$ is a homeomorphism from $\Omega$ onto $K$. Thus $\Phi$ is clearly the restriction of a linear isometry from $C(K)$ onto $C(\Omega)$.
\end{corollary}

\section{Phase-isometries between $S_{C(K)}$ and a Banach space $Y$}
\label{sec:3}

In this section, we will study phase-isometries between the unit spheres of $C(K)$ and an arbitrary Banach space $Y$. Although the result of such study is already known for isometries (see \cite[Theorem 3.2]{F} with $K$ being a compact metric space and \cite[Corollary 6]{L} with $K$ being a compact Hausdorff space), the techniques we use for phase-isometries are quite different from that for isometries.

For a Banach space $Y$, let $\mathcal{F}_Y$ be the set consisting of the maximal convex subsets of $S_Y$, and let $\mathcal{F}_Y^+$ be a subset of $\mathcal{F}_Y$ such that for every $F\in \mathcal{F}_Y$, either $F$ or $-F$ is in $\mathcal{F}_Y^+$.

 The proof of the following result proceeds along the same lines as the proof of
Proposition \ref{prop:6}, but the details are more complicated.

\begin{proposition}\label{prop:1}
Let $Y$ be a Banach space, and let $K$ be a compact Hausdorff space. Suppose that $T:S_{C(K)}\rightarrow S_Y$ is a surjective phase-isometry. Then there is a bijection $\Phi: K\rightarrow \mathcal{F}_Y^+$ such that
for every $t\in K$, we have $T(F_t \cup -F_t)=\Phi(t)\cup -\Phi(t)$.
\end{proposition}
\begin{proof}
The desired conclusion is equivalent to requiring that for every maximal convex subset $F\subset S_Y$, there exists an $s\in K$ such
that
\begin{equation}\label{equ:7}
T(F_s \cup -F_s)=F\cup-F.
\end{equation}
Indeed, if \eqref{equ:7} holds, then this $s$ is unique. Since $T$ is injective, every $F_t$ must has the form \eqref{equ:7} as dersired. Thus the required bijection $\Phi:K\rightarrow \mathcal{F}^+$ is simply defined by $\Phi(t)=F$ if $F\in \mathcal{F}^+$ or $\Phi(t)=-F$ if $F\in \mathcal{F}^+$ such that \eqref{equ:7} (with $s$ instead of $t$) holds.

For \eqref{equ:7}, we will consider $T$'s converse $T^{-1}$. Write $V:=T^{-1}$. Then $V$ is phase-isometry from $S_Y$ onto $S_{C(K)}$, and given $x\in F$, set
\begin{equation*}
K_x:=\{s\in K: |V(x)(s)|=1\}.
\end{equation*}
We will show that the closed sets $\{K_x: x\in F\}$ have the finite intersection property. To this end,
 for any $x_1,x_2,\cdots,x_n\in F$, set $A=\mbox{co}\{x_1,x_2,\cdots,x_n\}$. A similar argument as in Proposition \ref{prop:6} gives us an $x_0\in S_Y$ such that
$$\{V(x_1), V(x_2),\cdots, V(x_n)\}\subset \mbox{st}\big(V(x_0)\big)\cup -\mbox{st}\big(V(x_0)\big),$$
and $\mbox{st}\big(V(x_0)\big)$ and $-\mbox{st}\big(V(x_0)\big)$ are maximal convex subsets of $S_{Y_0}$ with $Y_0$ being some separable subspace of $C(K)$. It follows that there is a $t_0\in\Omega$ such that $|V(x_i)(t_0)|=1$ for all $i=1,2,\cdots,n$. Hence the closed sets $\{K_x:x\in F\}$ have the finite intersection property as desired.  Moreover since $K$ is compact, there is at least one $s\in K$ such that $|V(x)(s)|=1$ for all $x\in F$.  This therefore proves that $V(F)\subset F_s\cup -F_s$. By Lemma \ref{lemma:1}, we have $V(-F)\subset F_s\cup -F_s$ as well. Namely, we have
 \begin{equation}\label{equ:9}
   F\cup -F\subset T(F_s\cup -F_s).
 \end{equation}
 For the converse, assume that this is not true, i.e., $F\cup -F\varsubsetneq T(F_s\cup -F_s)$. We can choose $z_2\in T(F_s\cup -F_s)$ with $z_2\notin F\cup -F$. This implies that there are $y_1\in F$ and $y_2\in -F$ such that
 \begin{equation*}
   \|z_2+y_1\|<2\quad \mbox{and} \quad \|z_2+y_2\|<2.
 \end{equation*}
 Set $u_1:=(y_1-y_2)/2\in F$. \eqref{equ:9} ensures us the existence of $f_1, f_2 \in F_s\cup -F_s$ such that
 \begin{equation*}
   T(f_1)=u_1\quad\mbox{and} \quad T(f_2)=z_2.
 \end{equation*}
 Then
 \begin{align*}
  2&=\|f_1+f_2\|\vee \|f_1-f_2\|\\&=\|T(f_1)-T(f_2)\| \vee \|T(f_1)+T(f_2)\|=\|u_1-z_2\|\vee\|u_1+z_2\|<2.
 \end{align*}
A contradiction thus finishes the proof.
\end{proof}
We begin to describe the class of functionals $x^*$ in the unit sphere such that $\{y\in B_Y: x^*(y)=1\}=\Phi(t)$ and
we want to illustrate what happens to those vectors in $S_Y$ under the image of these $x^*$.
\begin{lemma}\label{lemma:8}
Let $Y$ be a Banach space, and let $K$ be a compact Hausdorff space. Suppose that $T:S_{C(K)}\rightarrow S_Y$ is a surjective phase-isometry and $\Phi$ is a bijection as in Proposition \ref{prop:1}. Then for every $t\in K$, there is a functional $x^*\in S_{Y^*}$ such that $x^*(x)=1$ for all $x\in\Phi(t)$ and $x^*(T(f))\in\{f(t),-f(t)\}$ for all $f\in S_{C(K)}$.
\end{lemma}
\begin{proof}
Given $t\in K$, Proposition \ref{prop:1} establishes that $T(F_t\cup -F_t)=\Phi(t)\cup -\Phi(t)$. Let $x^*\in S_{Y^*}$ satisfy
\begin{equation*}
  \Phi(t):=\{x\in B_Y: x^*(x)=1\}.
\end{equation*}
We will prove that $x^*$ has the required properties. To this end, we will construct another phase-isometry with better property instead of $T$.

Fix an $f_0\in F_t$. Considering $f_0$ instead of $-f_0$ if necessary, we may assume that $T(f_0)\in \Phi(t)$. By Proposition \ref{prop:1} we can define a map $T_t: F_t\rightarrow \Phi(t)$ by $T_{t}(g)=T(g)$ if $T(g)\in \Phi(t)$ and $T_{t}(g)=-T(g)$ if $T(g)\in -\Phi(t)$.
By this we continue to define another map $\overline{T_t}:S_{C(K)}\rightarrow S_Y$ by
$$\overline{T_{t}}(g)= \left
\{ \begin{array}{ll}
T_t(g) & \mbox{if }\, g\in F_t; \\
-T_t(-g) & \mbox{if }\,g\in -F_t; \\
T(g)  & \mbox{if } \, g\notin -F_t\cup F_t.
\end{array}
\right.$$
It is obvious that $\overline{T_{t}}$ is phase-equivalent to $T$, and hence $\overline{T_{t}}$ is also a surjective phase-isometry such that $\overline{T_{t}}(F_t)=\Phi(t)$ and $\overline{T_{t}}(-F_t)=-\Phi(t)$ which follows from the definition and Proposition \ref{prop:1}.

Now for every $f\in S_{C(K)}$, if $|f(t)|=1$, the desired conclusion follows straight from Proposition \ref{prop:1}.  We only need to deal with the case of $|f(t)|<1$. Since $\overline{T_{t}}$ is odd by Lemma \ref{lemma:2}, it suffices to consider those $f$ satisfying $f(t)\geq0$. In other words, we may assume that $0\leq f(t)<1$. As shown in the front part of the proof of Theorem \ref{lemma:5}, for every $\varepsilon>0$, there is a $\psi,\phi\in F_t$ such that
\begin{equation*}
 \| f-\psi\|\leq1-f(t)+\varepsilon,
\end{equation*}
and $\|f+\phi\|\vee\|f-\phi\| <1+f(t)+\varepsilon$.
Since $\overline{T_{t}}$ is a phase-isometry, we see that
\begin{align*}
  1-f(t)+\varepsilon\geq \|f-\psi\|\wedge \|f+\psi\|&=\|\overline{T_{t}}(f)-\overline{T_{t}}(\psi)\|\wedge\|\overline{T_{t}}(f)+\overline{T_{t}}(\psi)\|\\
  &\geq\Big(1-x^*\big(T(f)\big)\Big)\wedge\Big(1+x^*\big(T(f)\big)\Big).
\end{align*}
This leads directly to $x^*(\overline{T_{t}}(f))\geq f(t)$ or $-x^*(\overline{T_{t}}(f))\geq f(t)$.

If $x^*(\overline{T_{t}}(f))\geq f(t)>0$, note that
\begin{align*}
  1+x^*(\overline{T_{t}}(f))&\leq \inf_{g\in -F_t}\|\overline{T_{t}}(f)-\overline{T_{t}}(g)\| \\
                            &\leq \inf_{g\in -F_t}(\|f-g\|\vee \|f+g\|)\\
                            &\leq\|f-\phi\|\vee \|f+\phi\|\leq1+f(t)+\varepsilon.
\end{align*}
Thus the arbitrariness of $\varepsilon$ implies that $x^*(\overline{T_{t}}(f))\leq f(t)$, and as a consequence, equality $x^*(\overline{T_{t}}(f))= f(t)$ is valid. If $-x^*(\overline{T_{t}}(f))\geq f(t)$, a similar discussion gives $-x^*(\overline{T_{t}}(f))=f(t)$. Since $\overline{T_{t}}(f)=T(f)$ by the definition, this finishes the proof.
\end{proof}
For a surjective phase-isometry $T: S_{C(K)}\rightarrow S_Y$, we let $G$ denote the subset of $S_{Y^*}$ that consists of those $x^*$ satisfying Lemma \ref{lemma:8}, that is
\begin{equation*}
  G:=\{x^*\in S_{Y^*}: x^*(T(f))\in\{f(t),-f(t)\} \, \mbox{for some}\, t\in K \, \mbox{and for all}\, f\in S_{C(K)}\}.
\end{equation*}
Note that $G$ is symmetric, i.e., $-G=G$. To simplify the notation, we set
\begin{equation*}
G^+:=\{x^*\in G: x^*(T(\mathds{1}_K))=1\},
\end{equation*}
where $\mathds{1}_K$ is the constant function with value 1 on $K$. Then
it is easily checked that for every $x^*\in G$, either $x^*$ itself or $-x^*$ is in $G^+$. 

One can see from the above lemma that for every $x^*\in G^+,$ the sign of $x^*(T(f))$ may depend on $f$ and some $t\in K$. However, we will prove that it only depends on $f$, and more, the sign of one $x^*\in G^+$ on some $T(f)$ decides all those of $z^*\in G^+$ on this $T(f)$.
\begin{lemma}\label{lemma:10}
Let $Y$ be a Banach space, and let $K$ be a compact Hausdorff space. Suppose that $T:S_{C(K)}\rightarrow S_Y$ is a surjective phase-isometry.
For every $f\in S_{C(K)}$ and every $x^*\in G^+$, if $x^*(T(f))=\theta f(t)$ for some $t\in K$ and some $\theta\in\{-1,1\}$, then $z^*(T(f))=\theta f(s)$ for all $s\in K$ and all $z^*\in G^+$.
\end{lemma}
\begin{proof}
We first show that the conclusion holds for all $f\in S_{C(K)}^+\cup S_{C(K)}^-$.  Given $f\in S_{C(K)}^+$, if $\|T(f)-T(\mathds{1}_K)\|=\|f-\mathds{1}_K\|\leq1$, then $$x^*(T(\mathds{1}_K)-T(f))\leq \|T(f)-T(\mathds{1}_K)\|\leq1$$ for all $x^*\in G^+$. It follows that $x^*(T(f))\geq 0$ for all $x^*\in G^+$. Similarly we have  $x^*(T(f))\leq 0$ for all $x^*\in G^+$ if $\|T(f)+T(\mathds{1}_K)\|=\|f-\mathds{1}_K\|\leq1.$ These combined with Lemma \ref{lemma:8} proves the desired conclusion in the case of $f\in S_{C(K)^+}$. Another case of $f\in S_{C(K)^-}$ follows from an analogous argument.

Following a proof which is similar in
spirit to that of Theorem \ref{lemma:5}, we conclude that desired conclusion holds for every $f\in S_{C(K)}$ and those $t\in K$ with $|f(t)|=1$.
Suppose that the conclusion are not true for all $f\in S_{C(K)}$. This means that we can find an $f_0\in S_{C(K)}\setminus (S_{C(K)}^+\cup S_{C(K)}^-)$ and $x^*_1,x^*_2\in G^+$ such that
\begin{equation*}
x^*_1(T(f_0))=f_0(t_1)\neq 0 \quad \mbox{and} \quad x^*_2(T(f_0))=-f_0(t_2)\neq 0.
\end{equation*}
 We will apply the method in Theorem \ref{lemma:5} to prove that there is only one sign $\theta$ for a fixed $f\in S_{C(K)}.$
Write $a=|f_0(t_1)|\wedge|f_0(t_2)|$. Observe that $f_0\notin S_{C(K)}^+\cup S_{C(K)}^-$. Hence the set $A:=\{t\in K: f_0(t)\geq -\frac{a}{3}\}$ cannot be empty. If the set $B=\{t\in K: f_0(t)\leq -\frac{a}{2}\}$ is also not empty, by Urysohn's lemma there is a $g:K\rightarrow [-1,1]$ such that $g(t)=1$ for $t\in A$, and $g(t)=-1$ if $t\in B$.  If $B=\emptyset$, take $g$ to be the constant function $\mathds{1}_K$. Put $x:=T(f_0)-T(g)$ and $y:=T(f_0)+T(g)$. The previous argument implies that
\begin{align*}
1+\frac{a}{2}\geq \|f_0-g\|&\geq\|x\|\wedge \|y\| \\
                   &\geq (|x^*_1(x)|\vee|x^*_2(x)|)\wedge (|x^*_2(y)|\vee |x^*_2(y)|)\\
                   &\geq 1+|f_0(t_1)|\wedge|f_0(t_2)|=1+a.
           \end{align*}
A contradiction therefore finishes the proof.
\end{proof}

We are now ready for our second main result. The previous lemma will play a key role in the proof.
\begin{theorem}\label{thm2}
Let $Y$ be a Banach space, and let $K$ be a compact Hausdorff space. Suppose that $T:S_{C(K)}\rightarrow S_Y$ is a surjective phase-isometry. Then $T$ is phase-equivalent to an isometry which can be extended to a linear isometry from $C(K)$ onto $Y$.
\end{theorem}
\begin{proof}
We first define the natural positive homogeneous extension of $T$ written as $\widetilde{T}: C(K)\rightarrow Y$ by
$$\widetilde{T}(f)= \left
\{ \begin{array}{ll}
\|f\|T(\frac{f}{\|f\|}) & \mbox{if }\, f\neq 0; \\
0 & \mbox{if }\,f=0.
\end{array}
\right.$$
We now take up to construct the desired isometry from $C(K)$ onto $Y$. Fix an $x^*_t\in G^+$ concerning some $t\in K$. Define $U:C(K)\rightarrow Y$ by
$$U(f)= \left
\{ \begin{array}{ll}
\widetilde{T}(f) & \mbox{if }\, x^*_t(\widetilde{T}(f))=f(t); \\
-\widetilde{T}(f)   & \mbox{if }\,x^*_t(\widetilde{T}(f))=-f(t)\neq 0;\\
0   &\mbox{if }\,f=0.
\end{array}
\right.$$
Then it is obvious that the restriction of $U$ in the unit sphere $S_{C(K)}$ is phase-equivalent to $T$ and by the definition, the identity
\begin{equation}\label{equ:10}
x_t^*(U(f))=f(t)
\end{equation}
holds for all $f\in C(K)$. Thus it follows from Lemma \ref{lemma:10} that $x^*(U(f))=f(s)$ for all $x^*\in G^+$ and all $s\in K$. Note first that for all $f,g\in C(K)$, there is an $x_0^*\in G^+$ such that $|x_0^*(U(f)-U(g))|= \|U(f)-U(g)\|$.
We deduce from this and \eqref{equ:10} that
\begin{align*}
 \|U(f)-U(g)\|&=\max\{|x^*\big(U(f)-U(g)\big)|: x^*\in G^+\}\\&=\max\{|f(t)-g(t)|: t\in K\}\\&=\|f-g\|.
\end{align*}
So $U$ is an isometry, and its definition shows that $U$ is surjective. Finally, the well-known Mazur-Ulam proves that $U$ is linear. The proof is complete.
\end{proof}
For a measure space $(\Omega, \Sigma, \mu)$, denote by
$L^\infty(\mu)$  the space of all
measurable essentially bounded functions $f$ with the essential
supremum norm
\begin{align*}
\|f\|=\mbox{ess. sup}_{t\in\Omega}|f(t)|.
\end{align*}
We have a similar result for $L^\infty(\mu)$ by means of Theorem \ref{thm2}.
\begin{corollary}
Let $Y$ be a Banach space, and let $(\Omega, \Sigma, \mu)$ be a measure space. Suppose that $T:S_{L^{\infty}(\mu)}\rightarrow S_Y$ is a surjective phase-isometry. Then $T$ is phase-equivalent to an isometry which can be extended to a linear isometry from $L^{\infty}(\mu)$ onto $Y$.
\end{corollary}
\begin{proof}
Since $L^\infty(\mu)$ is a unital commutative $C^*$-algebra, we can find a compact Hausdorff
space $K$ such that $L^\infty(\mu)\cong C(K)$. The desired conclusion follows directly from this and Theorem \ref{thm2}.
\end{proof}


%
%



\end{document}